\documentclass[12pt]{article}

\usepackage[utf8]{inputenc}
\usepackage[T1]{fontenc}
\usepackage{amsmath,amssymb,amsthm}
\usepackage{mathtools}
\usepackage{enumerate}
\usepackage[margin=2.5cm]{geometry}
\usepackage{hyperref}
\usepackage{graphicx}
\usepackage{booktabs}

\newtheorem{theorem}{Theorem}[section]
\newtheorem{proposition}[theorem]{Proposition}
\newtheorem{lemma}[theorem]{Lemma}
\newtheorem{corollary}[theorem]{Corollary}

\theoremstyle{definition}
\newtheorem{definition}[theorem]{Definition}
\newtheorem{example}[theorem]{Example}
\theoremstyle{remark}
\newtheorem{remark}[theorem]{Remark}

\newcommand{\One}{\textsc{One}}
\newcommand{\All}{\textsc{All}}

\title{Optimal strategies in the all-heads coin game}
\author{Peter Pfaffelhuber\thanks{%
  Albert-Ludwigs-Universit\"at Freiburg, Freiburg, Germany.
  Email: \href{mailto:p.p@stochastik.uni-freiburg.de}{p.p@stochastik.uni-freiburg.de}.
  ORCID: \href{https://orcid.org/0000-0002-6421-5460}{0000-0002-6421-5460}.}}
\date{\today}

\begin{document}
\maketitle

\begin{abstract}
  We study a sequential coin-flipping game: a player starts with
  $n$~coins, each heads with probability~$p$, and in each round flips
  all remaining coins and must set aside at least one head, losing if
  none shows.  The player wins once all coins have been set aside.
  The optimal winning probability~$w_{n,p}$ obeys a Bellman equation
  with a nonlinear suffix-maximum operator.  For $p=\tfrac12$ every
  strategy achieves $w_{n,1/2}=\tfrac12$.  For $p>\tfrac12$ the
  strategy~\One{} (set aside a single head) is optimal,
  $n\mapsto w_{n,p}$ is strictly increasing, and the limit
  $W(p):=\lim_n w_{n,p}$ has an explicit series representation with
  $p\le W(p)<1$.  For $p<\tfrac12$ near~$\tfrac12$ we give a
  first-order perturbation expansion in $\delta:=\tfrac12-p$: the
  deficit satisfies
  $\tfrac12-w_{n,\,1/2-\delta}\approx\delta\,c_n$, where $c_n$ obeys a
  linear recursion for $n\ge7$ with limit $L\approx1.7035$.  To first
  order the optimal-value sequence has a strict local minimum at
  $n=5$ and no local maximum.
\end{abstract}



\section{Introduction}\label{sec:intro}

\paragraph{A note on authorship, and disclosure of generative AI use.}
All material accompanying this paper is collected in the public
repository \emph{coins} at \url{https://github.com/pfaffelh/coins};
the file and directory names used below
(\texttt{simulation/}, \texttt{CoinsLean/}, and so on) refer to that
repository.
In accordance with Taylor \& Francis policy, the author discloses that
a generative AI tool---Anthropic's Claude (versions Opus~4.6,
Opus~4.7, and Opus~4.8)---was used, interactively and in conversations
with the author, to produce the mathematical text, the numerical code
(\texttt{simulation/}), and the complete Lean~4/Mathlib formalization
presented here.  To be precise about attribution: the underlying
mathematical ideas, the selection of the research question
(alternative objective $w_{n,p}$ vs.\ $v_{n,p}$), the decision to
carry out a first-order perturbation in $\delta=\tfrac12-p$, the
structuring of \S\ref{sec:below} as a joint induction, and the
decision to formally verify the paper are all the author's.
Claude's role was execution: drafting the exposition from the
author's outline, proposing and debugging proof tactics in Lean,
selecting Mathlib lemmas, spotting where a generic approach had to
be replaced by a case-by-case one (for instance the cubic-bound
inequality in \S\ref{sec:joint}, or the Tannery-style dominated
convergence argument in the proof of Theorem~\ref{thm:limit}), and
producing the numerical scripts.  The reason for using the tool was
thus to accelerate drafting, formalization, and numerical exploration
under the author's direction; individual decisions in both categories
were settled through back-and-forth dialogue, and the author reviewed
every edit.  The generative AI tool is not, and cannot be, listed as
an author, since it cannot take responsibility for the work or manage
copyright and licensing; the author takes full responsibility for the
content, including any remaining errors.  This responsibility is
backed by an unusually strong safeguard: the correctness of every
numbered result is independently machine-checked by the
Lean~4/Mathlib formalization, which uses no \texttt{sorry} and only
the three standard foundational axioms \texttt{propext},
\texttt{Classical.choice}, and \texttt{Quot.sound}.

The headline statements that have been formalized are collected,
together with the seven shared definitions on which they depend, in
\href{https://github.com/pfaffelh/coins/blob/main/CoinsLean/Challenge.lean}{\texttt{CoinsLean/\allowbreak Challenge.lean}}
and
\href{https://github.com/pfaffelh/coins/blob/main/CoinsLean/CoinsLean/Defs.lean}{\texttt{CoinsLean/\allowbreak CoinsLean/\allowbreak Defs.lean}};
together these two files are the entire trust surface of the
formalization.  A referee can then independently verify the proofs
mechanically by running the Lean comparator~\cite{LeanComparator}
on the public repository --- see Appendix~\ref{app:lean} for
instructions.  Readers who want the full record can consult a
prompt-by-prompt transcript at
\href{https://github.com/pfaffelh/coins/blob/main/journal.md}{\texttt{journal.md}};
a more granular development history of the formalization appears
in Appendix~\ref{app:lean}.

\subsection{The game}

Fix a probability $p\in(0,1)$ and a positive integer~$n$.  A player
is given $n$~coins, each of which independently lands heads with
probability~$p$ and tails with probability $q:=1-p$.  The game
proceeds in rounds.  In each round:
\begin{enumerate}[(i)]
\item All remaining coins are flipped simultaneously.
\item If no coin shows heads, the game is lost.
\item Otherwise the player must set aside at least one coin that
  shows heads.  The set-aside coins leave the game permanently.
\item If every coin has been set aside, the game is won.
  Otherwise, a new round begins with the remaining coins.
\end{enumerate}
The player's objective is to maximise the probability of winning.

Two natural strategies present themselves:
\begin{itemize}
\item \textbf{Strategy~\One:} set aside exactly one head when
  $1\le k\le n-1$ heads appear, and set aside all~$n$ when $k=n$.
\item \textbf{Strategy~\All:} set aside every head that appears
  ($i=k$).
\end{itemize}
Strategy~\One{} is conservative: it preserves as many coins as
possible for future rounds, betting that the one set-aside coin has
already ``done its job.''  Strategy~\All{} is greedy: it cashes in
every success immediately.  Neither is universally optimal.

\paragraph{Relation to prior work.}
\sloppy A close variant of this game --- building on an earlier
question posed by Joachim~Breitner on Mathematics Stack
Exchange~\cite{Breitner2015} --- was introduced and studied by
van~Doorn~\cite{vanDoorn2024}, who took as objective the
\emph{expected number of heads} collected, $v_{n,p}$, rather than
the all-heads-wins probability $w_{n,p}$ studied here.  Van~Doorn's
analysis is essentially complete for $p\ge\tfrac12$: he proves
that there exist absolute constants $\phi=(\sqrt5-1)/2\approx
0.618$ and $p_0\approx 0.5495$ such that the optimal strategy
follows~\One{} (in his notation, ``A'') for $p\ge\phi$, switches
between~\One{} and~\All{} at a level $n(p)$ for $p_0<p<\phi$, and
follows~\All{} (``B'') throughout for $\tfrac12\le p\le p_0$.  For
$p<\tfrac12$, however, van~Doorn writes that ``everything is a lot
less clear'' and restricts himself to numerical conjectures and a
non-monotonicity result.

The present paper makes two changes in emphasis.  First, the
objective: maximising the all-heads probability $w_{n,p}$ rather
than the expected count $v_{n,p}$. As a consequence, the $p>\tfrac12$ case becomes simpler: 
the strategy~\One{} is always optimal --- a fact noted in passing
in~\cite[\S 10]{vanDoorn2024} and proved here in
Theorem~\ref{thm:above} below with a self-contained joint
induction. Second, and more importantly, our focus is the regime
$p<\tfrac12$ that van~Doorn explicitly leaves open.  We carry out
a first-order perturbation in $\delta=\tfrac12-p$ and obtain a
\emph{closed-form description} of the leading behaviour: the
deficit coefficient $c_n$ satisfies a linear recursion
(Proposition~\ref{prop:linear}), the limit $L=\lim c_n$ admits an
explicit formula and a numerical value $L\approx
1.7035$ (Theorem~\ref{thm:limit}), and the local-extremum
structure of $n\mapsto w_{n,p}$ is fully determined to first order
(Corollary~\ref{cor:shape}).

A third, methodological, change concerns formal verification.
Since the results of the present paper were generated in collaboration with AI,
we want to avoid hallucinations: the reader does not have to be a specialist in Lean~4, 
the corresponding proof assistant, but has to acknowledge the fact that (i) the recursive definition
for $w_{n,p}$ from \eqref{eq:bellman} is correctly implemented and (ii) the results as they stand 
compile without error, meaning that Lean~4 has checked every proof. 
More precisely, every numbered result
of the present paper --- including the perturbation analysis of
\S\ref{sec:below} --- has been formally verified in Lean~4 using
the Mathlib library, with no \texttt{sorry}s, no custom axioms,
and no \texttt{native\_decide} or \texttt{unsafe} blocks.  A
line-by-line correspondence between the manuscript and the formal
proofs is provided in Appendix~\ref{app:lean}.

\subsection{Markov decision processes and the Bellman equation}

The game is a finite-state \emph{Markov decision process} (MDP): the
state is the number~$m$ of coins remaining, the action is the number
$i\in\{1,\dots,k\}$ of heads to set aside (where $k$ is the random
number of heads), and the transition is
$m\mapsto m-i$.  There are two absorbing states: \emph{win}
($m=0$, all coins collected) and \emph{lose} (reached when $k=0$ in
some round).  A \emph{policy} (or \emph{strategy})~$\pi$ specifies,
for each state~$n$ and each observed number~$k\in\{1,\dots,n\}$ of
heads, the number $i_\pi(n,k)\in\{1,\dots,k\}$ of heads to set
aside; we write $v_n^\pi$ for the resulting winning probability
starting from state~$n$.

Because every policy absorbs with probability at least
$p^m+(1-p)^m>0$ per round, the MDP is a \emph{stochastic
shortest-path} (SSP) problem in the sense of Bertsekas and Tsitsiklis
\cite{BertsekisTsitsiklis1991}: the Bellman operator is a contraction
and the optimal value function is its unique fixed point.

\begin{definition}[Optimal winning probability]\label{def:w}
  Let $w_{0,p}:=1$.  For $n\ge 1$ define
  \begin{equation}\label{eq:bellman}
    w_{n,p} \;:=\; p^n + \sum_{j=1}^{n-1}\binom{n}{j}p^{n-j}q^{j}\,
    \max_{j\le m\le n-1} w_{m,p},
  \end{equation}
  where $q=1-p$ and the empty sum (at $n=1$) gives
  $w_{1,p}=p$.
\end{definition}

The term $p^n$ accounts for the event that all $n$~coins land heads
(an immediate win).  Each summand with $j\ge 1$ accounts for the
event that exactly~$j$ coins land tails; the player then chooses to
keep some number $m\in\{j,\dots,n-1\}$ of coins (setting aside
$n-m$ heads), and the game continues from state~$m$ with optimal
value~$w_{m,p}$.

The Bellman equation~\eqref{eq:bellman} is the standard optimality
equation for this SSP--MDP.  For general background on Markov decision
processes and the Bellman equation we refer to the monographs of
Puterman~\cite{Puterman1994}, Bertsekas~\cite{Bertsekas2017}, and
Stokey, Lucas and Prescott~\cite{StokeyLucas1989}.

\subsection{Strategies \One{} and \All}

Under strategy~\One{} the player always keeps $m=n-1$ coins
(except when $k=n$), so the Bellman recursion collapses to
\begin{equation}\label{eq:a-rec}
  a_{n,p} \;=\; p^n + \bigl(1-p^n-(1-p)^n\bigr)\,a_{n-1,p},
  \qquad a_{0,p}:=1.
\end{equation}
Under strategy~\All{} the player always keeps $m=j$ coins (the
number of tails), giving
\begin{equation}\label{eq:b-rec}
  b_{n,p} \;=\; p^n + \sum_{j=1}^{n-1}\binom{n}{j}p^{n-j}q^{j}\,
  b_{j,p},
  \qquad b_{0,p}:=1.
\end{equation}
Note that the recursions for both~$a$ and~$b$ are \emph{linear}
(no $\max$ operator), which makes them analytically more tractable
than the Bellman equation~\eqref{eq:bellman}.

\subsection{Outline}

In Section~\ref{sec:fair} we prove that at $p=\tfrac12$ the optimal
winning probability is $w_{n,1/2}=\tfrac12$ for every~$n$, and that
in fact every reasonable strategy achieves this value.  In
Section~\ref{sec:above} we show that for $p>\tfrac12$ the
strategy~\One{} is always optimal and $n\mapsto w_{n,p}$ is strictly
increasing; as a consequence (\S\ref{sec:above-limit}) the limit
$W(p)=\lim_n w_{n,p}$ exists and admits a closed-form series
representation.  Section~\ref{sec:below} treats the case
$p<\tfrac12$ near~$\tfrac12$: we introduce a first-order
perturbation in $\delta=\tfrac12-p$, derive a recursion for the
leading coefficient~$c_n$, prove that it simplifies to a linear
recursion for $n\ge 7$, and compute the limit $L=\lim c_n$.
Section~\ref{sec:numerics} complements the first-order analysis
(and so restricts to $p<\tfrac12$) with numerical experiments that
illustrate the shape of $n\mapsto w_{n,p}$ away from $p=\tfrac12$,
including the emergence of local maxima for $p$ bounded away
from~$\tfrac12$ (invisible in the first-order expansion), and
Section~\ref{sec:discussion}
collects concluding observations and open questions.
Appendix~\ref{app:lean}
describes the formal verification of all results in Lean~4 and
provides a line-level map between the manuscript and the
formalization on GitHub.

\section{The fair coin: \texorpdfstring{$p=\tfrac12$}{p = 1/2}}
\label{sec:fair}

\begin{theorem}\label{thm:half}
  For every $n\ge 1$, $w_{n,1/2}=\tfrac12$.  More generally, for any
  policy~$\pi$ satisfying $i_\pi(n,n)=n$ for every~$n$ (i.e.\ the
  player accepts the immediate win when all coins land heads), the
  winning probability is $v_n^\pi=\tfrac12$.

  In particular, both strategies~\One{} and~\All{} achieve the
  optimum: $a_{n,1/2}=b_{n,1/2}=w_{n,1/2}=\tfrac12$.
\end{theorem}

\begin{proof}
  We use the binomial identity
  \begin{equation}\label{eq:bin-half}
    \sum_{j=1}^{n-1}\binom{n}{j}\Bigl(\frac{1}{2}\Bigr)^{n}
    \;=\; 1-2^{1-n}.
  \end{equation}

  \emph{Step~1.}  We show $w_{n,1/2}=\tfrac12$ by strong induction.
  For $n=1$: $w_{1,1/2}=\tfrac12$.  For $n\ge 2$, suppose
  $w_{m,1/2}=\tfrac12$ for all $1\le m\le n-1$.  Then every
  suffix-maximum equals~$\tfrac12$, so
  \begin{align*}
    w_{n,1/2}
    &= 2^{-n} + \frac{1}{2}\sum_{j=1}^{n-1}\binom{n}{j}2^{-n} \\
    &= 2^{-n} + \tfrac{1}{2}\bigl(1-2^{1-n}\bigr) \\
    &= 2^{-n} + \tfrac{1}{2} - 2^{-n}
    \;=\; \tfrac{1}{2}.
  \end{align*}

  \emph{Step~2.}  For a general policy~$\pi$ with
  $i_\pi(n,n)=n$, conditioning on the number of heads
  $K\sim\mathrm{Bin}(n,\tfrac12)$ gives
  \[
    v_n^\pi = 2^{-n}\cdot 1 + \sum_{k=1}^{n-1}\binom{n}{k}2^{-n}\,
    v_{n-i_\pi(n,k)}^\pi.
  \]
  Since $1\le i_\pi(n,k)\le k\le n-1$ for $k<n$, the subscript
  $n-i_\pi(n,k)$ lies in $\{1,\dots,n-1\}$, and by strong induction
  each $v_{n-i_\pi(n,k)}^\pi=\tfrac12$.  The same calculation as in
  Step~1 gives $v_n^\pi=\tfrac12$.
\end{proof}

\begin{remark}\label{rem:refuse}
  A policy that refuses the immediate win at $k=n$ (i.e.\
  $i_\pi(n,n)<n$) achieves strictly less than~$\tfrac12$: the
  contribution of the $k=n$ event drops from $2^{-n}$ to
  $2^{-n}\cdot v_{n-i_\pi(n,n)}^\pi < 2^{-n}$, while all other terms
  remain at most $\tfrac12\cdot\binom{n}{k}2^{-n}$, giving
  $v_n^\pi<\tfrac12$.
\end{remark}

\section{Above the fair coin: \texorpdfstring{$p>\tfrac12$}{p > 1/2}}
\label{sec:above}

\begin{theorem}\label{thm:above}
  For $p>\tfrac12$ and all $n\ge 2$:
  \begin{enumerate}[\rm(i)]
  \item $w_{n,p}>w_{n-1,p}$, i.e.\ $n\mapsto w_{n,p}$ is strictly
    increasing.
  \item $w_{n-1,p}<\dfrac{p^n}{p^n+q^n}$.
  \item Strategy~\One{} is optimal: $w_{n,p}=a_{n,p}$.
  \end{enumerate}
\end{theorem}

\begin{proof}
  Note that $n\mapsto \dfrac{p^n}{p^n+q^n} = \dfrac{1}{1+(q/p)^n}$ is increasing for $p \in (\tfrac 12, 1)$. 
  We prove (i) and (ii) simultaneously by strong induction on~$n$.

  \emph{Base case $n=2$.}
  One computes $w_{1,p}=p$ and
  $w_{2,p}=p^2+(1-p^2-q^2)p = 3p^2-2p^3$.  Then
  $w_{2,p}-w_{1,p}=p(1-p)(2p-1)>0$ since $p >\tfrac12$.
  Also $w_{1,p}=p = p/(p + q) < p^2/(p^2+q^2)$.

  \emph{Inductive step.}  Assume (i) and (ii) hold for all indices up
  to~$n-1$.  Since $n\mapsto w_{n,p}$ is increasing on
  $\{1,\dots,n-1\}$ by hypothesis, the suffix-maximum
  $\max_{j\le m\le n-1}w_{m,p}=w_{n-1,p}$ for every
  $j\in\{1,\dots,n-1\}$.  Hence the Bellman
  equation~\eqref{eq:bellman} reduces to
  \[
    w_{n,p} = p^n + (1-p^n-q^n)\,w_{n-1,p},
  \]
  which is exactly the recursion~\eqref{eq:a-rec} for
  strategy~\One.  This proves~(iii).

  For~(i), $w_{n,p}>w_{n-1,p}$ iff
  $p^n(1-w_{n-1,p})>q^n w_{n-1,p}$, i.e.\
  $w_{n-1,p}<p^n/(p^n+q^n)$, which is hypothesis~(ii) at level
  $n-1$.

  For~(ii) at level~$n$, we need $w_{n,p}<p^{n+1}/(p^{n+1}+q^{n+1})$.
  Using the recursion and the inductive bound on~$w_{n-1,p}$, a
  direct computation shows this reduces to the inequality
  $(q/p)^n<(q/p)^{n-1}$, which holds since $q/p<1$.
\end{proof}

\subsection{The limit \texorpdfstring{$W(p)=\lim_{n\to\infty}w_{n,p}$}{W(p) = lim w(n,p)}}\label{sec:above-limit}

The collapse of the Bellman suffix-maximum, which is the central
content of Theorem~\ref{thm:above}, has a useful by-product: the
Bellman recursion for $w_{n,p}$ is genuinely \emph{linear} for
every $n\ge 1$.

\begin{corollary}[Linear recursion above $\tfrac12$]\label{cor:above-lin}
  For $p>\tfrac12$ and every $n\ge 1$,
  \begin{equation}\label{eq:above-lin}
    w_{n,p} \;=\; p^n + \bigl(1-p^n-q^n\bigr)\,w_{n-1,p},
  \end{equation}
  with $w_{0,p}=1$.
\end{corollary}

\begin{proof}
  For $n=1$ the sum in~\eqref{eq:bellman} is empty and
  $w_{1,p}=p=p^1+(1-p-q)\cdot 1$ since $p+q=1$.  For $n\ge 2$,
  combine the suffix-max collapse from the proof of
  Theorem~\ref{thm:above} with the binomial identity
  $\sum_{j=1}^{n-1}\binom{n}{j}p^{n-j}q^j=1-p^n-q^n$.
\end{proof}

The linear recursion~\eqref{eq:above-lin} is the analogue, for
$p>\tfrac12$, of the linear $c_n$-recursion of
Proposition~\ref{prop:linear} below.  It is structurally
\emph{simpler}: it holds for every $n\ge 1$ (no preliminary
joint-induction step is needed), the coefficients are explicit
polynomials in $p$, and the bounded monotone sequence converges
trivially by the Weierstrass criterion.

\begin{theorem}[Limit and explicit formula above $\tfrac12$]
  \label{thm:above-limit}
  For every $p\in(\tfrac12,1)$ the limit
  $W(p):=\lim_{n\to\infty}w_{n,p}$ exists, satisfies
  $p\le W(p)<1$, and admits the closed-form representation
  \begin{equation}\label{eq:above-W-formula}
    W(p) \;=\; \sum_{k=1}^\infty p^k\!\!\!\prod_{j=k+1}^\infty\!\!
            \bigl(1-p^j-q^j\bigr).
  \end{equation}
  The series converges absolutely at geometric rate.
\end{theorem}

\begin{proof}
  Existence follows from Theorem~\ref{thm:above}~(i) (monotone,
  bounded above by~$1$).  The lower bound $W(p)\ge w_{1,p}=p$ is
  immediate from monotonicity; combined with $p>\tfrac12$ this
  gives $W(p)>\tfrac12$.

  For the explicit formula, iterate~\eqref{eq:above-lin} from
  $k=1$ up to $k=n$:
  \begin{equation}\label{eq:W-finite-iter}
    w_{n,p} \;=\; \underbrace{\prod_{k=1}^n\!\bigl(1-p^k-q^k\bigr)
                               \cdot w_{0,p}}_{=\,0}
    \;+\; \sum_{k=1}^n p^k\!\prod_{j=k+1}^n\!\bigl(1-p^j-q^j\bigr),
  \end{equation}
  where the first product vanishes because its $k=1$ factor is
  $1-p-q=0$.  For $j\ge 2$ the factors
  $1-p^j-q^j=\sum_{i=1}^{j-1}\binom{j}{i}p^{j-i}q^i$ lie in $(0,1)$
  (positive from the $i=1$ term; strictly less than $1$ since we
  subtracted $p^j+q^j>0$).  Together with
  $\sum_{j\ge 2}(p^j+q^j)<\infty$, the partial products
  $\prod_{j=k+1}^n(1-p^j-q^j)$ converge as $n\to\infty$ to a
  strictly positive limit $\Pi_k\in(0,1)$, and the resulting
  series $\sum_k p^k\Pi_k$ converges absolutely by comparison with
  $\sum_k p^k$.  By Tannery's theorem
  (cf.~Knopp~\cite[\S 41]{Knopp1990}) with dominator $p^k$, limit
  and summation commute, giving~\eqref{eq:above-W-formula}.

  It remains to show $W(p)<1$.  Set $U(p):=1-W(p)$ and
  $u_n:=1-w_{n,p}$; subtracting~\eqref{eq:above-lin} from~$1$ and
  using $p^n+q^n+(1-p^n-q^n)=1$ gives the companion recursion
  $u_n=q^n+(1-p^n-q^n)u_{n-1}$ with $u_0=0$.
  Iterating identically to~\eqref{eq:W-finite-iter} gives the
  analogue
  \[
    U(p) \;=\; \sum_{k=1}^\infty q^k\Pi_k,
  \]
  which is a strictly positive sum (every term is positive: $q>0$
  and $\Pi_k>0$).  Hence $W(p)=1-U(p)<1$.
\end{proof}

\begin{remark}[Parallel with Theorem~\ref{thm:limit}]\label{rem:above-vs-below}
  The iterated identity~\eqref{eq:W-finite-iter} and its limit
  formula~\eqref{eq:above-W-formula} have the same structure as
  the finite iteration~\eqref{eq:cn-finite-iter} and the
  formula~\eqref{eq:L-formula} for~$L$ below: in both cases a
  linear recursion of the form $x_n=A_n+(1-B_n)\,x_{n-1}$ is
  iterated and Tannery's theorem identifies the resulting limit
  as $\lim x_{n_0-1}\prod(1-B_j)+\sum A_k\prod(1-B_j)$, with
  dominator $\sum|A_k|<\infty$.  The two differ only in the
  coefficients: for $W(p)$, $A_k=p^k$ and $B_k=p^k+q^k$; for
  $L$, $A_k=A_{\mathrm{lin}}(k)$ and $B_k=B_{\mathrm{lin}}(k)$
  from Proposition~\ref{prop:linear}, and the structural step
  leading to the linear recursion is trivial above $\tfrac12$
  (immediate from monotonicity) but requires the joint induction
  of \S\ref{sec:joint} below.
\end{remark}

\begin{figure}[h]
  \centering
  \includegraphics[width=0.78\linewidth]{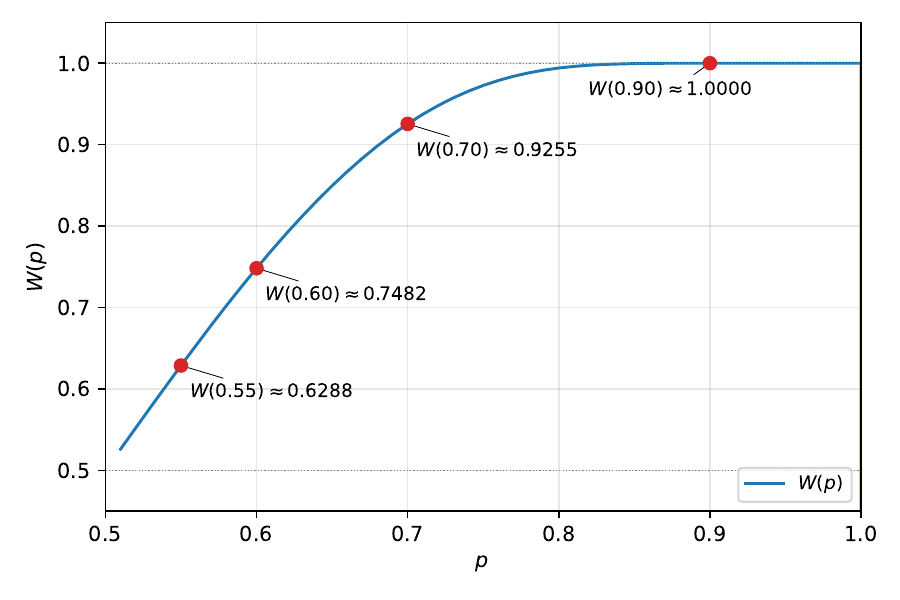}
  \caption{The limit $W(p)=\lim_{n\to\infty}w_{n,p}$ as a function
    of~$p$, for $p\in(\tfrac12,1)$.  Computed via the recursion
    of Corollary~\ref{cor:above-lin} iterated to $n=400$ in
    $80$-digit arithmetic.  Selected values:
    $W(0.55)\approx 0.6288$, $W(0.6)\approx 0.7482$,
    $W(0.7)\approx 0.9255$, $W(0.9)\approx 0.99998$.  The limit
    is monotone in~$p$, bounded below by the identity line (since
    $W(p)\ge p$) and strictly below~$1$ by
    Theorem~\ref{thm:above-limit}.}
  \label{fig:W-above}
\end{figure}

\section{Below the fair coin: perturbation in
  \texorpdfstring{$\delta=\tfrac12-p$}{delta = 1/2 - p}}
\label{sec:below}

For $p<\tfrac12$ the structure of $n\mapsto w_{n,p}$ is far richer:
numerical experiments show that the sequence develops local minima
and maxima, and the optimal strategy is neither~\One{}
nor~\All{} in general; see \cite{vanDoorn2024}, Section~10, and
Section~\ref{sec:numerics} of the present paper for illustrations.

To get analytical traction, we study the regime
$p=\tfrac12-\delta$ with $\delta>0$ small.

\subsection{The deficit and its recursion}

\begin{definition}\label{def:deficit}
  For $p\in(0,1)$ and $n\ge 1$, define the
  \emph{deficit}
  $\Delta_{n,p}:=\tfrac12-w_{n,p}$.
\end{definition}

By Theorem~\ref{thm:half}, $\Delta_{n,1/2}=0$ for every $n\ge 1$.
At $n=0$ the convention $w_{0,p}:=1$ (``no coins left'' is a
trivial win) yields $\Delta_{0,p}=-\tfrac12$ for every~$p$; this
value will reappear in Proposition~\ref{prop:delta-rec}.

\begin{proposition}\label{prop:delta-rec}
  The deficit satisfies
  \begin{equation}\label{eq:delta-rec}
    \Delta_{n,p} \;=\;
    \frac{q^n-p^n}{2}
    + \sum_{j=1}^{n-1}\binom{n}{j}p^{n-j}q^{j}\,
    \min_{j\le m\le n-1}\Delta_{m,p},
  \end{equation}
  with $\Delta_{0,p}=-\tfrac12$.
  In particular:
  \begin{enumerate}[\rm(i)]
  \item For $p<\tfrac12$: $\Delta_{n,p}>0$ for every $n\ge 1$
    (i.e.\ $w_{n,p}<\tfrac12$).
  \item For $p>\tfrac12$: $\Delta_{n,p}<0$ for every $n\ge 1$
    (i.e.\ $w_{n,p}>\tfrac12$).
  \end{enumerate}
\end{proposition}

\begin{proof}
  Substitute $w_{m,p}=\tfrac12-\Delta_{m,p}$ into the
  Bellman equation~\eqref{eq:bellman}.  Using
  $\max_m(\tfrac12-\Delta_m) = \tfrac12 - \min_m\Delta_m$ and the
  binomial identity $\sum_{j=1}^{n-1}\binom{n}{j}p^{n-j}q^j
  = (p+q)^n - p^n - q^n = 1 - p^n - q^n$:
  \begin{align*}
    \tfrac12 - \Delta_{n,p}
    &= p^n + \sum_{j=1}^{n-1}\binom{n}{j}p^{n-j}q^j
       \Bigl(\tfrac12 - \min_{j\le m\le n-1}\Delta_{m,p}\Bigr) \\
    &= p^n + \tfrac12\bigl(1 - p^n - q^n\bigr)
       - \sum_{j=1}^{n-1}\binom{n}{j}p^{n-j}q^j\min_{j\le m\le n-1}\Delta_{m,p} \\
    &= \tfrac12 + \tfrac{p^n - q^n}{2}
       - \sum_{j=1}^{n-1}\binom{n}{j}p^{n-j}q^j\min_{j\le m\le n-1}\Delta_{m,p},
  \end{align*}
  and rearranging gives~\eqref{eq:delta-rec}.
  The sign claims follow by induction using $q^n-p^n>0$ (resp.\ $<0$)
  and the positivity (resp.\ negativity) of the $\min$ terms.
\end{proof}

\subsection{First-order expansion}

Set $p=\tfrac12-\delta$ with $\delta>0$ small.  Then
$q=\tfrac12+\delta$, and the constant term in~\eqref{eq:delta-rec} is
\[
  \frac{q^n-p^n}{2}
  = \frac{(\tfrac12+\delta)^n-(\tfrac12-\delta)^n}{2}
  = \frac{n\delta}{2^{n-1}} + O(\delta^3).
\]
The binomial weights become $\binom{n}{j}2^{-n}+O(\delta)$, and the
$\min$ terms are $O(\delta)$ by Proposition~\ref{prop:delta-rec}.
Thus $\Delta_{n,p}=O(\delta)$.

\begin{proposition}[First-order coefficient]\label{prop:cn}
  Write $\Delta_{n,p}=c_n\delta+O(\delta^2)$ as $\delta\to 0^+$.
  Then $c_1=1$ and for $n\ge 2$:
  \begin{equation}\label{eq:cn-rec}
    c_n \;=\; \frac{n}{2^{n-1}}
    + \frac{1}{2^n}\sum_{j=1}^{n-1}\binom{n}{j}\,
    \min_{j\le m\le n-1}c_m.
  \end{equation}
\end{proposition}

\begin{proof}
  Match first-order coefficients in~\eqref{eq:delta-rec}.  The
  constant term contributes $n/2^{n-1}$.  Each binomial weight is
  $\binom{n}{j}2^{-n}+O(\delta)$; the $\min$ of the
  $\Delta_m=c_m\delta+O(\delta^2)$ at first order is
  $(\min c_m)\delta+O(\delta^2)$.  Products of $O(\delta)$ terms are
  $O(\delta^2)$.
\end{proof}

\begin{example}\label{ex:cn-small}
  The first few values of~$c_n$:
  \begin{gather*}
    c_1=1,\quad c_2=\tfrac{3}{2},\quad
    c_3=\tfrac{27}{16}=1.6875,\quad
    c_4=\tfrac{111}{64}\approx 1.7343, \\
    c_5=\tfrac{3555}{2048}\approx 1.7358,\quad
    c_6=\tfrac{113337}{65536}\approx 1.7293.
  \end{gather*}
  The sequence increases on $\{1,\dots,5\}$ and decreases on $\{5,6,7,...\}$ (as we will prove below).
  In particular $c_5>c_4 > c_6$.
  Each of the six identities above is formally verified in
  Lean~4 (\href{https://github.com/pfaffelh/coins/blob/c60bcd3/CoinsLean/CoinsLean/Perturbation.lean\#L76}{\texttt{c\_one}},
  \dots,
  \href{https://github.com/pfaffelh/coins/blob/c60bcd3/CoinsLean/CoinsLean/Perturbation.lean\#L244}{\texttt{c\_six}});
  see Appendix~\ref{app:lean}, Table~\ref{tab:lean-map}, for direct
  links to the source.
  Figure~\ref{fig:cn-sequence} plots the sequence up to $n=25$,
  together with the limit $L=\lim_{n\to\infty}c_n\approx
  1.7035$ and the uniform lower bound $\tfrac{27}{16}=1.6875$
  established in Lemma~\ref{lem:lower}.

  \begin{figure}[h]
    \centering
    \includegraphics[width=0.78\linewidth]{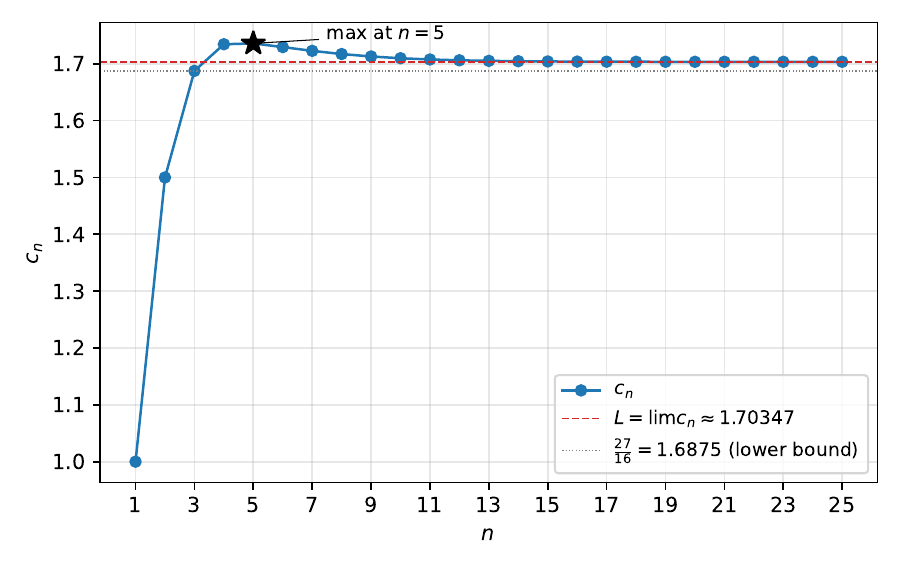}
    \caption{The sequence $n\mapsto c_n$ for $1\le n\le 25$.
      Computed via the recursion of
      Proposition~\ref{prop:cn} in $80$-digit rational
      arithmetic.  The sequence increases on
      $\{1,\dots,5\}$ (with maximum at $n=5$, marked by a star),
      strictly decreases for $n\ge 5$, and converges from above
      to $L\approx 1.7035$.  The lower bound $\tfrac{27}{16}$
      from Lemma~\ref{lem:lower} is shown as a dotted line.}
    \label{fig:cn-sequence}
  \end{figure}
\end{example}

\subsection{Collapse, lower bound, and linear recursion}\label{sec:joint}

The four results below — the asymptotic collapse of the $\min$
in~\eqref{eq:cn-rec}~(Lemma~\ref{lem:collapse}), the uniform lower
bound~$c_n\ge\tfrac{27}{16}$ for $n\geq 4$~(Lemma~\ref{lem:lower}), the strict
decrease of $(c_n)_{n\ge 5}$~(Lemma~\ref{lem:decreasing}), and the
linear recursion for $n\ge 7$~(Proposition~\ref{prop:linear}) — are
mutually entangled: each, taken alone, depends on parts of the
others.  Stating them in isolation produces an apparent circular
dependency.  We therefore state all four results first and then
prove them simultaneously by induction on~$n$, with the IH at level
$<n$ supplying exactly what is needed at level~$n$.

\begin{lemma}[Asymptotic collapse]\label{lem:collapse}
  For all $n\ge 7$ and $j\in\{1,\dots,n-1\}$:
  \begin{equation}\label{eq:collapse}
    \min_{j\le m\le n-1}c_m
    \;=\;
    \begin{cases}
      c_j     & \text{if } j\in\{1,2,3\},\\
      c_{n-1} & \text{if } j\in\{4,\dots,n-1\}.
    \end{cases}
  \end{equation}
\end{lemma}

\begin{lemma}\label{lem:lower}
  $c_n\ge\tfrac{27}{16}$ for every $n\ge 4$.
\end{lemma}

\begin{lemma}\label{lem:decreasing}
  The sequence $(c_n)_{n\ge 5}$ is strictly decreasing.
\end{lemma}

\begin{proposition}[Linear recursion]\label{prop:linear}
  For $n\ge 7$:
  \begin{equation}\label{eq:cn-linear}
    c_n \;=\; A_n + (1-B_n)\,c_{n-1},
  \end{equation}
  where
  \begin{equation}\label{eq:AB}
    A_n := \frac{n}{2^{n-1}}
    + \frac{n\,c_1+\binom{n}{2}c_2+\binom{n}{3}c_3}{2^n},
    \qquad
    B_n := \frac{2+n+\binom{n}{2}+\binom{n}{3}}{2^n}.
  \end{equation}
  Moreover the algebraic identity
  \begin{equation}\label{eq:alg-id}
    A_n - \tfrac{27}{16}B_n
    \;=\; -\frac{3(n^2-15n+36)}{32\cdot 2^n}
  \end{equation}
  holds for every~$n\ge 1$ (independently of the linear recursion).
\end{proposition}

\begin{proof}[Joint proof of Lemmas~\ref{lem:collapse},
  \ref{lem:lower}, \ref{lem:decreasing} and Proposition~\ref{prop:linear}]
  The identity~\eqref{eq:alg-id} is a direct algebraic computation
  from~\eqref{eq:AB} using $\binom{n}{2}=\tfrac{n(n-1)}{2}$ and
  $\binom{n}{3}=\tfrac{n(n-1)(n-2)}{6}$; we use it freely below.
  Formally, it is verified in Lean as the lemma
  \href{https://github.com/pfaffelh/coins/blob/c60bcd3/CoinsLean/CoinsLean/Perturbation.lean\#L932}{\texttt{alg\_id}};
  see Appendix~\ref{app:lean}, Table~\ref{tab:lean-map}.

  Let $P(n)$ denote the conjunction of the following
  four claims (each understood vacuously when its index range is
  empty):
  \begin{enumerate}[\rm(a)]
  \item Lemma~\ref{lem:collapse} holds at level~$n$ (when $n\ge 7$);
  \item $c_n\ge\tfrac{27}{16}$ (when $n\ge 4$);
  \item $c_n<c_{n-1}$ (when $n\ge 6$);
  \item the linear recursion~\eqref{eq:cn-linear} holds at level~$n$
    (when $n\ge 7$).
  \end{enumerate}
  We prove $P(n)$ for all $n\ge 4$ by strong induction on~$n$.

  \medskip
  \emph{Base cases $n\in\{4,5,6\}$.}  By Example~\ref{ex:cn-small},
  $c_4=\tfrac{111}{64}$, $c_5=\tfrac{3555}{2048}$,
  $c_6=\tfrac{113337}{65536}$, all of which exceed
  $\tfrac{27}{16}=\tfrac{108}{64}$, so~(b) holds.  Also $c_5>c_6$,
  so~(c) holds at $n=6$.  Claims~(a) and~(d) are vacuous for
  $n\le 6$.

  \medskip
  \emph{Inductive step.}  Fix $n\ge 7$ and assume $P(m)$ for all
  $4\le m<n$.  In particular: by IH(b),
  $c_m\ge\tfrac{27}{16}$ for $4\le m\le n-1$; by IH(c),
  $c_5>c_6>\dots>c_{n-1}$.  We now establish~(a)--(d) at level~$n$.

  \emph{(a) Collapse.}  For $j\in\{1,2,3\}$: from
  $c_1=1<c_2=\tfrac32<c_3=\tfrac{27}{16}$ (Example~\ref{ex:cn-small})
  and $c_m\ge\tfrac{27}{16}=c_3$ for $m\ge 4$ (IH(b)), the minimum
  over $\{j,\dots,n-1\}$ is~$c_j$.

  For $j\in\{4,\dots,n-1\}$: iterating IH(c) gives the chain
  $c_5>c_6>\dots>c_{n-1}$, so in particular
  $c_m\ge c_{n-1}$ for every $m\in\{5,\dots,n-1\}$.  For $m=4$,
  $c_4=\tfrac{111}{64}>c_6$ (Example~\ref{ex:cn-small}); using the
  segment $c_6\ge c_{n-1}$ of the same chain, we obtain $c_4\ge
  c_{n-1}$ as well.  Hence the minimum over $\{j,\dots,n-1\}$ for
  $j\ge 4$ is~$c_{n-1}$.

  \emph{(d) Linear recursion.}  Substituting the collapse~(a)
  just proved into~\eqref{eq:cn-rec} and splitting the sum at $j=3$
  yields~\eqref{eq:cn-linear}--\eqref{eq:AB}.  Formally, this
  substitution is verified in Lean as the theorem
  \href{https://github.com/pfaffelh/coins/blob/c60bcd3/CoinsLean/CoinsLean/Perturbation.lean\#L1495}{\texttt{c\_linear\_rec}};
  see Appendix~\ref{app:lean}, Table~\ref{tab:lean-map}.

  \emph{(b) Lower bound.}  From~(d) and~\eqref{eq:alg-id},
  \[
    c_n - \tfrac{27}{16}
    \;=\; \bigl(A_n-\tfrac{27}{16}B_n\bigr)
        + (1-B_n)\bigl(c_{n-1}-\tfrac{27}{16}\bigr).
  \]
  By IH(b), $c_{n-1}-\tfrac{27}{16}\ge 0$.  Note also
  $0<B_n<1$ for all $n\ge 7$: the numerator
  $2+n+\binom{n}{2}+\binom{n}{3}$ in~\eqref{eq:AB} is strictly
  less than $2^n=\sum_{j=0}^n\binom{n}{j}$, the difference being
  $\sum_{j=4}^{n-1}\binom{n}{j}>0$.  For $n\in\{7,\dots,12\}$, the
  quadratic $n^2-15n+36$ is non-positive on $\{3,\dots,12\}$
  (its roots are $3$ and $12$), so $A_n-\tfrac{27}{16}B_n\ge 0$
  by~\eqref{eq:alg-id}, and~(b) follows.  For $n\ge 13$ we argue
  cumulatively: define $\varepsilon_n:=c_n-\tfrac{27}{16}$ and
  $\delta_n:=3(n^2-15n+36)/(32\cdot 2^n)$,
  so~(d) gives $\varepsilon_n=(1-B_n)\varepsilon_{n-1}-\delta_n$
  for every $n\ge 13$.  Iterating this identity from level~$12$
  upwards and using the trivial bound $0<1-B_k\le 1$ yields, for
  every $n\ge 13$,
  \begin{equation}\label{eq:eps-iter}
    \varepsilon_n
    \;=\; \prod_{k=13}^{n}(1-B_k)\,\varepsilon_{12}
    \;-\; \sum_{k=13}^{n}\!\Bigl(\prod_{j=k+1}^{n}(1-B_j)\Bigr)\delta_k
    \;\ge\; \Bigl(\prod_{k=13}^{n}(1-B_k)\Bigr)\varepsilon_{12}
            \;-\; \sum_{k=13}^{n}\delta_k.
  \end{equation}
  The starting value $\varepsilon_{12}>\tfrac{1}{60}$ is verified by
  exact rational computation (formally:
  \href{https://github.com/pfaffelh/coins/blob/c60bcd3/CoinsLean/CoinsLean/Perturbation.lean\#L2995}{\texttt{c\_twelve\_buffer}});
  the cumulative negative contribution is bounded by
  $\sum_{n\ge 13}\delta_n \le\tfrac{297}{65536}<\tfrac{1}{200}$
  (using $\sum_{n\ge 13}n^2/2^n=\tfrac{99}{2048}$; formally:
  \href{https://github.com/pfaffelh/coins/blob/c60bcd3/CoinsLean/CoinsLean/Perturbation.lean\#L889}{\texttt{delta\_tail\_bound}});
  meanwhile $\prod_{n\ge 13}(1-B_n)>\tfrac{7}{8}$ (using
  $\sum_{n\ge 13}B_n<\tfrac{1}{8}$; formally:
  \href{https://github.com/pfaffelh/coins/blob/c60bcd3/CoinsLean/CoinsLean/Perturbation.lean\#L864}{\texttt{B\_tail\_bound}}).
  Inserting these three inequalities into~\eqref{eq:eps-iter} gives
  $\varepsilon_n>\tfrac{1}{60}\cdot\tfrac{7}{8}-\tfrac{1}{200}
  =\tfrac{23}{2400}>0$ for every $n\ge 13$.  The iterated inequality
  itself is formally verified (in the slightly weaker
  Bernoulli-style form
  $\prod(1-B_k)\ge 1-\sum B_k$) as
  \href{https://github.com/pfaffelh/coins/blob/c60bcd3/CoinsLean/CoinsLean/Perturbation.lean\#L949}{\texttt{cum\_eps\_bound}}.

  (\emph{Note for an inductive treatment:} the sequence
  $(\delta_n)_{n\ge 13}$ is not monotone — one computes
  $\delta_{14}/\delta_{13}=\tfrac{11}{10}$ — but its ratio
  decreases through~$1$ at~$n=14$ and tends to~$\tfrac12$.  An
  iterated proof of the bound on $\sum_{n\ge 13}\delta_n$ should
  therefore split at~$n=14$; the closed-form route above sidesteps
  this subtlety.)

  \emph{(c) Strict decrease.}  By~(d), $c_n-c_{n-1}=A_n-B_nc_{n-1}$.
  Using $c_{n-1}\ge\tfrac{27}{16}$ (IH(b)) and $B_n>0$, this is
  $\le A_n-\tfrac{27}{16}B_n<0$ for $n\ge 13$ by~\eqref{eq:alg-id}
  ($n^2-15n+36>0$ there).  For $n\in\{7,\dots,12\}$, $c_n<c_{n-1}$
  is verified by direct rational computation (chain via the linear
  recursion from $c_6$ up to $c_n$, then numerical comparison).
  This completes the inductive step.  The combined statement of~(c)
  for all $n\ge 5$ is exposed as
  \href{https://github.com/pfaffelh/coins/blob/c60bcd3/CoinsLean/CoinsLean/Perturbation.lean\#L1488}{\texttt{c\_strict\_anti\_from\_five}};
  the corresponding statement of~(b) for all $m\ge 4$ as
  \href{https://github.com/pfaffelh/coins/blob/c60bcd3/CoinsLean/CoinsLean/Perturbation.lean\#L1483}{\texttt{c\_ge\_27\_16\_full}}.
\end{proof}

\subsection{The limit}

Since $(c_n)_{n\ge 5}$ is strictly decreasing
(Lemma~\ref{lem:decreasing}) and bounded below by $\tfrac{27}{16}$
(Lemma~\ref{lem:lower}), the Weierstra\ss{} monotone-convergence
criterion gives at once:

\begin{corollary}[Existence of the limit]\label{cor:L-exists}
  The limit $L:=\lim_{n\to\infty}c_n$ exists and satisfies
  $\tfrac{27}{16}\le L\le c_5=\tfrac{3555}{2048}$.
\end{corollary}

\noindent
Formally verified as
\href{https://github.com/pfaffelh/coins/blob/c60bcd3/CoinsLean/CoinsLean/Perturbation.lean\#L1517}{\texttt{c\_limit\_exists}}.
The actual mathematical content of this subsection is the
\emph{explicit} formula for~$L$:

\begin{theorem}[Explicit formula for the limit]\label{thm:limit}
  For any $n_0\ge 7$,
  \begin{equation}\label{eq:L-formula}
    L \;=\; c_{n_0-1}\prod_{m\ge n_0}(1-B_m)
    \;+\; \sum_{k\ge n_0}A_k\prod_{m>k}(1-B_m).
  \end{equation}
  Both the infinite product and the series converge at geometric
  rate since $A_n,B_n=O(n^3/2^n)$.  Numerically,
  \begin{equation}\label{eq:L-value}
    L \;=\; 1.70347176087173673645\ldots
  \end{equation}
\end{theorem}

\begin{proof}
  Fix $n_0\ge 7$.  Iterating the linear
  recursion~\eqref{eq:cn-linear} from $n_0$ up to~$n$ yields, for
  every $n\ge n_0$,
  \begin{equation}\label{eq:cn-finite-iter}
    c_n \;=\; c_{n_0-1}\prod_{m=n_0}^n(1-B_m)
        \;+\; \sum_{k=n_0}^n A_k\prod_{m=k+1}^n(1-B_m).
  \end{equation}
  By Corollary~\ref{cor:L-exists} we already know that the
  left-hand side converges to~$L$; the content of the theorem is
  that the right-hand side converges to the asserted explicit
  expression.  We pass to the limit $n\to\infty$ in each summand.

  \emph{Step 1 (infinite product).}  Since $B_m=O(m^3/2^m)$, the
  series $\sum_{m\ge n_0}B_m$ converges absolutely.  Because
  $0<B_m<1$ for all $m\ge 7$ (shown in the inductive proof of
  the previous subsection), the partial products
  $\prod_{m=n_0}^n(1-B_m)$ converge as $n\to\infty$ to a strictly
  positive limit $\prod_{m\ge n_0}(1-B_m)$; cf.\ Knopp
  \cite[§29]{Knopp1990}.

  \emph{Step 2 (series of products).}  For each fixed $k\ge n_0$
  the same argument gives
  $\prod_{m=k+1}^n(1-B_m)\to\prod_{m>k}(1-B_m)$ as $n\to\infty$.
  Moreover every factor lies in $(0,1)$, so each finite product
  lies in $(0,1]$, and
  \begin{equation}\label{eq:series-dominate}
    \Bigl|A_k\prod_{m=k+1}^n(1-B_m)\Bigr|\;\le\;|A_k|
    \qquad(n\ge k).
  \end{equation}
  Since $A_k=O(k^3/2^k)$, the series $\sum_{k\ge n_0}|A_k|$
  converges.  The bound~\eqref{eq:series-dominate} therefore supplies
  a summable dominant, so by Tannery's theorem — the dominated
  convergence theorem for series, see Knopp~\cite[§41]{Knopp1990} —
  limit and summation commute:
  \[
    \lim_{n\to\infty}\sum_{k=n_0}^n A_k\prod_{m=k+1}^n(1-B_m)
    \;=\; \sum_{k\ge n_0}A_k\prod_{m>k}(1-B_m).
  \]

  \emph{Conclusion.}  Taking $n\to\infty$ in~\eqref{eq:cn-finite-iter}
  and using Steps~1 and~2 gives~\eqref{eq:L-formula}.  Both the
  product and the series converge at geometric rate
  since $A_n,B_n=O(n^3/2^n)$.
\end{proof}

\subsection{Shape of the optimal-value sequence near
  \texorpdfstring{$p=\tfrac12$}{p=1/2}}

\begin{corollary}\label{cor:shape}
  For $\delta>0$ sufficiently small and $p=\tfrac12-\delta$:
  \begin{enumerate}[\rm(i)]
  \item $w_{n-1,p}-w_{n,p}=(c_n-c_{n-1})\delta+O(\delta^2)$.
  \item The sequence $n\mapsto w_{n,p}$ has, to first order
    in~$\delta$, a strict local minimum at $n=5$ (since $c_5>c_4$
    and $c_5>c_6$).
  \item There is no local maximum at first order: $c_n$ is eventually
    decreasing, so $w_{n,p}$ is eventually increasing.
  \end{enumerate}
\end{corollary}

\begin{remark}
  Numerical computation shows that the sequence $n\mapsto w_{n,p}$
  does develop strict local maxima for $p$ bounded away
  from~$\tfrac12$ (e.g.\ at $n=9$ for $p=0.42$, verified by exact
  rational arithmetic).  This is consistent with the corollary: local
  maxima are a \emph{non-perturbative} phenomenon invisible to the
  first-order expansion.  Van Doorn~\cite{vanDoorn2024}, Section~10,
  conjectures that the optimal strategy for $p<\tfrac12$ is
  ``infinitely complex.''
\end{remark}

\section{Numerical experiments for \texorpdfstring{$p<\tfrac12$}{p<1/2}}\label{sec:numerics}

The first-order analysis of Section~\ref{sec:below} describes
the shape of $n\mapsto w_{n,p}$ only in a shrinking neighbourhood
of $p=\tfrac12$.  Corollary~\ref{cor:shape} guarantees, for every
sufficiently small $\delta>0$, a unique local minimum at $n=5$ and
no local maximum.  Numerical computation of $w_{n,p}$ confirms
this picture for $p$ close to $\tfrac12$, and shows that local
maxima emerge for $p$ bounded away from $\tfrac12$.

All values of $w_{n,p}$ below are computed by the exact Bellman
recursion~\eqref{eq:bellman}, evaluated in arbitrary-precision
rational (or $80$-decimal-digit) arithmetic via
\texttt{mpmath}.  The code is available in the \texttt{simulation/}
subdirectory of the repository.

Figure~\ref{fig:w_vs_n} plots $w_{n,p}$ against $n$ for five
values of $p$.  For $p=0.49$ (first-order regime) the sequence
bottoms out at $n=5$ and increases thereafter, in agreement with
Corollary~\ref{cor:shape}.  For $p=0.42$ and $p=0.45$ a
non-monotone structure emerges: the sequence first decreases,
passes through a local minimum, then rises to a local
\emph{maximum} before decreasing again and converging.

\begin{figure}[h]
  \centering
  \includegraphics[width=0.72\linewidth]{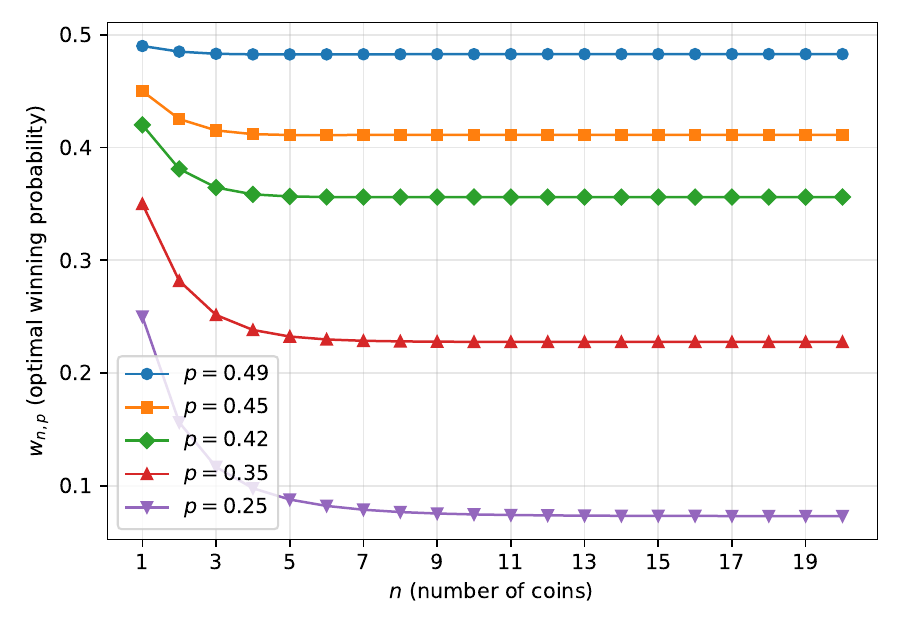}
  \caption{The optimal winning probability $w_{n,p}$ as a
    function of~$n$ for five values of $p<\tfrac12$.  Computed by
    the exact Bellman recursion in $80$-digit arithmetic.  For
    $p=0.49$ only the predicted local minimum at $n=5$ is visible;
    for $p=0.42$ and $p=0.45$ local maxima appear.  For $p=0.35$
    and $p=0.25$ no local extrema occur within $n\le 20$.}
  \label{fig:w_vs_n}
\end{figure}

The non-monotone behaviour is invisible in the plot at the scale
of Figure~\ref{fig:w_vs_n}, because the oscillations are small
compared with the total variation of $w_{n,p}$ in~$n$.
Figure~\ref{fig:local_max_detail} zooms in on $p=0.42$ and
$p=0.45$ and highlights the local minima and maxima found by
direct inspection of the sequence.

\begin{figure}[h]
  \centering
  \includegraphics[width=\linewidth]{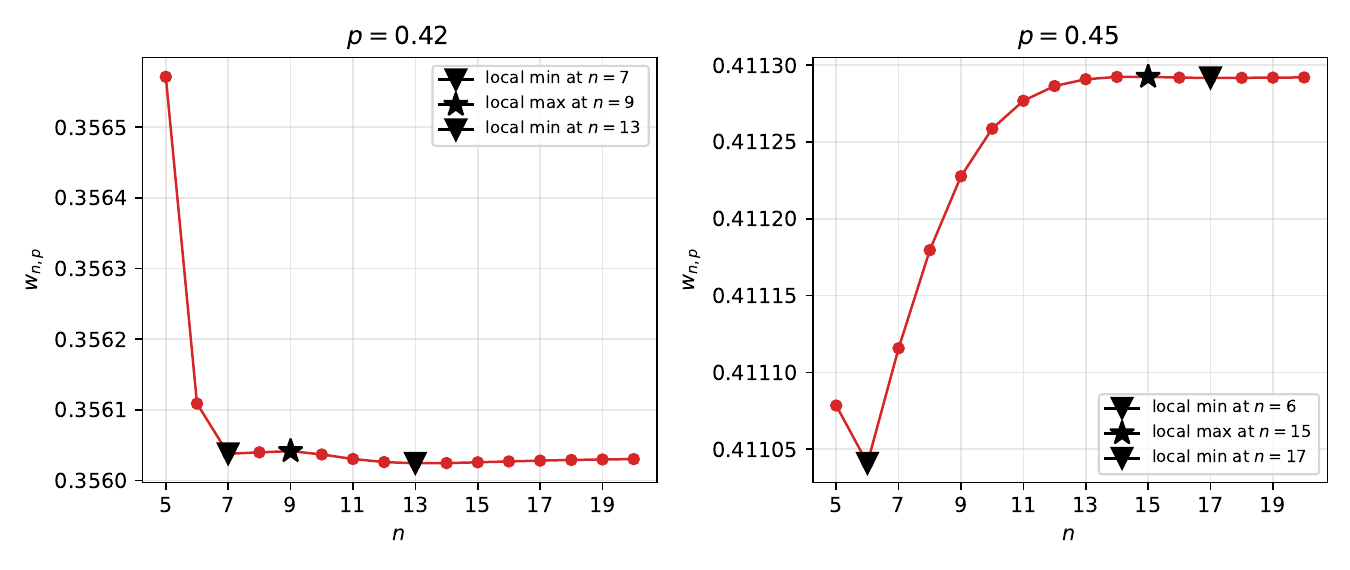}
  \caption{Zoomed-in view of $w_{n,p}$ for $p\in\{0.42,0.45\}$ and
    $n\in[5,20]$.  Stars mark local maxima, inverted triangles
    mark local minima.  At $p=0.42$ there is a local maximum at
    $n=9$; at $p=0.45$ there is a local maximum at $n=15$.}
  \label{fig:local_max_detail}
\end{figure}

Table~\ref{tab:w-values} lists the numerical values of~$w_{n,p}$
for $1\le n\le 20$ and the five values of~$p$ considered in
Figure~\ref{fig:w_vs_n}.  The local extrema are tabulated in
Table~\ref{tab:extrema}.

\begin{table}[h]
\centering
\small
\begin{tabular}{r *{5}{r}}
\toprule
$n$ & $p=0.49$ & $p=0.45$ & $p=0.42$ & $p=0.35$ & $p=0.25$ \\
\midrule
 1 & 0.49000000 & 0.45000000 & 0.42000000 & 0.35000000 & 0.25000000 \\
 2 & 0.48500200 & 0.42525000 & 0.38102400 & 0.28175000 & 0.15625000 \\
 3 & 0.48309103 & 0.41514272 & 0.36450348 & 0.25147259 & 0.11669922 \\
 4 & 0.48258194 & 0.41184019 & 0.35845021 & 0.23820137 & 0.09781647 \\
 5 & \textbf{0.48254059} & 0.41107840 & 0.35657129 & 0.23236999 & 0.08782906 \\
 6 & 0.48259223 & \textbf{0.41104082} & 0.35610901 & 0.22977285 & 0.08215652 \\
 7 & 0.48264193 & 0.41111568 & \textbf{0.35603793} & 0.22859504 & 0.07878468 \\
 8 & 0.48268806 & 0.41117954 & 0.35604004 & 0.22805232 & 0.07672145 \\
 9 & 0.48272442 & 0.41122762 & \textit{0.35604176} & 0.22779987 & 0.07543431 \\
10 & 0.48275111 & 0.41125863 & 0.35603705 & 0.22768230 & 0.07462024 \\
11 & 0.48276987 & 0.41127677 & 0.35603062 & 0.22762785 & 0.07410012 \\
12 & 0.48278262 & 0.41128638 & 0.35602626 & 0.22760288 & 0.07376520 \\
13 & 0.48279104 & 0.41129079 & \textbf{0.35602460} & 0.22759155 & 0.07354821 \\
14 & 0.48279647 & 0.41129231 & 0.35602484 & 0.22758643 & 0.07340694 \\
15 & 0.48279990 & \textit{0.41129237} & 0.35602596 & 0.22758409 & 0.07331459 \\
16 & 0.48280201 & 0.41129184 & 0.35602720 & 0.22758299 & 0.07325404 \\
17 & 0.48280329 & \textbf{0.41129168} & 0.35602833 & 0.22758245 & 0.07321422 \\
18 & 0.48280406 & 0.41129173 & 0.35602926 & 0.22758218 & 0.07318798 \\
19 & 0.48280451 & 0.41129185 & 0.35602997 & 0.22758204 & 0.07317065 \\
20 & 0.48280478 & 0.41129198 & 0.35603049 & 0.22758197 & 0.07315919 \\
\bottomrule
\end{tabular}
\caption{Numerical values of~$w_{n,p}$ for $n\in[1,20]$ and five
  values of $p<\tfrac12$.  Local minima are marked in
  \textbf{bold}, local maxima in \textit{italics}.}\label{tab:w-values}
\end{table}

\begin{table}[h]
\centering
\begin{tabular}{lll}
\toprule
$p$ & Local minima at $n=$ & Local maxima at $n=$ \\
\midrule
$0.49$ & $5$ & --- \\
$0.45$ & $6$, $17$ & $15$ \\
$0.42$ & $7$, $13$ & $9$ \\
$0.35$ & --- & --- \\
$0.25$ & --- & --- \\
\bottomrule
\end{tabular}
\caption{Local extrema of $n\mapsto w_{n,p}$ for
  $n\in[2,19]$.}\label{tab:extrema}
\end{table}

\section{Discussion}\label{sec:discussion}

For $p>\tfrac12$, the situation is much simpler than for
$p<\tfrac12$: the linear recursion of
Corollary~\ref{cor:above-lin} holds for every $n\ge 1$, and the
limit $W(p)=\lim_{n\to\infty}w_{n,p}$ admits the closed
form~\eqref{eq:above-W-formula}; numerically, $W(p)$ is a
smooth, strictly increasing function of $p\in(\tfrac12,1)$ with
$p\le W(p)<1$ (Figure~\ref{fig:W-above}).  The bulk of this
paper is therefore concerned with the more delicate regime
$p<\tfrac12$, where two features of the numerical data in
\S\ref{sec:numerics} deserve further comment.

\emph{(a) Local maxima are a non-perturbative phenomenon.}
Corollary~\ref{cor:shape}(iii) rules out local maxima of
$n\mapsto w_{n,p}$ \emph{to first order in $\delta=\tfrac12-p$}:
for any fixed~$n$, there is a $\delta_0(n)>0$ below which the
predicted sign of $c_{n+1}-c_n<0$ forces~$w_{n,p}<w_{n+1,p}$.
The local maximum at $n=9$ for $p=0.42$ ($\delta=0.08$), and the
one at $n=15$ for $p=0.45$ ($\delta=0.05$), correspond to the
threshold $\delta_0(n)$ being \emph{larger} than the actual
perturbation size at that~$n$, so higher-order terms dominate.
Numerically, $\delta_0(n)$ appears to decay rapidly in~$n$ (we
see no local maxima at all for $p\le 0.35$ within $n\le 20$),
which is consistent with the expected geometric decay of the
first-order constants~$A_n$ and~$B_n=O(n^3/2^n)$.

\emph{(b) The first-order picture holds robustly near $p=\tfrac12$.}
For $p=0.49$ ($\delta=0.01$), the only local extremum within
$n\le 20$ is the predicted minimum at $n=5$, with no spurious
oscillations.  This confirms numerically that the first-order
range-of-validity extends at least to $\delta=0.01$.  Pinning down
the size of $\delta_0(n)$ as a function of~$n$ remains open; we
conjecture that it decays roughly like $C/n^2$ for some absolute
constant~$C$.

\begin{remark}
  The picture for $p\le 0.35$ in our data --- a strictly
  decreasing sequence with no local extrema within $n\le 20$ ---
  is consistent with, but does not prove, monotone decrease on
  all of~$\mathbb{N}$.  Van~Doorn~\cite{vanDoorn2024} proves that
  strategy~\All{} is \emph{not} optimal for $p<\tfrac12$ (there
  exists $n$ with $w_{n,p}\ne b_{n,p}$); equivalently, the
  sequence $n\mapsto w_{n,p}$ is not monotone non-increasing.
  The failure of monotonicity is, however, not always detectable
  as a local extremum in~$n$; it can manifest as strict inequality
  $w_{n,p}>b_{n,p}$ without the sequence~$w_{n,p}$ itself having
  oscillations.  Investigating this gap is beyond the scope of the
  present paper.
\end{remark}

\section*{Acknowledgements}

The author thanks Joachim Breitner for introducing him to the
game as well as for help in Lean~4, Wouter van~Doorn for valuable discussion, 
and Harald Binder for suggesting this new kind of doing mathematics.
This research was supported by the DFG Project-ID 499552394–SFB 1597.

\section*{Disclosure statement}

No potential conflict of interest was reported by the author.

\section*{Data availability statement}

The Lean~4/Mathlib formalization and the numerical code that support
the findings of this study are openly available in the repository
\emph{coins} at \url{https://github.com/pfaffelh/coins}, and are
archived at Zenodo,
\url{https://doi.org/10.5281/zenodo.20489409} (concept DOI; the
version archived for this paper is
\href{https://doi.org/10.5281/zenodo.20489410}{10.5281/zenodo.20489410}).


\appendix

\section{Formal verification in Lean~4}\label{app:lean}

All results of this paper have been formally verified in
Lean~4 (toolchain \texttt{leanprover/lean4:v4.29.0}) using the
Mathlib library at release tag \texttt{v4.29.0} (revision
\texttt{8a178386}).  The formalization is publicly available at
\begin{center}
  \url{https://github.com/pfaffelh/coins}
\end{center}
and builds successfully against this Mathlib version with no open
goals.  The proofs use only the three standard Lean foundational
axioms (\texttt{propext}, \texttt{Classical.choice},
\texttt{Quot.sound}); no additional axioms, \texttt{sorry}s,
\texttt{native\_decide}, \texttt{unsafe} blocks, or
\texttt{opaque} definitions appear.

Table~\ref{tab:lean-map} lists every numbered theorem,
proposition, lemma, and corollary stated in the main text together
with its corresponding Lean theorem.  Each hyperlink points to the
exact line of the current pinned revision
(commit~\href{https://github.com/pfaffelh/coins/tree/c60bcd3}{\texttt{c60bcd3}}).

\renewcommand{\arraystretch}{1.15}
\begin{table}[h]
\centering
\begin{tabular}{@{}lll@{}}
\hline
\textbf{Manuscript} & \textbf{Lean theorem} & \textbf{Location} \\
\hline
Theorem~\ref{thm:half} &
  \texttt{w\_half} &
  \href{https://github.com/pfaffelh/coins/blob/c60bcd3/CoinsLean/CoinsLean/Optimal.lean\#L102}{Optimal.lean:102} \\
Theorem~\ref{thm:above} (joint) &
  \texttt{above\_half} &
  \href{https://github.com/pfaffelh/coins/blob/c60bcd3/CoinsLean/CoinsLean/Above.lean\#L108}{Above.lean:108} \\
Corollary~\ref{cor:above-lin} &
  \texttt{above\_linear\_rec} &
  \href{https://github.com/pfaffelh/coins/blob/c60bcd3/CoinsLean/CoinsLean/AboveLimit.lean\#L71}{AboveLimit.lean:71} \\
Theorem~\ref{thm:above-limit} (existence) &
  \texttt{above\_limit\_exists} &
  \href{https://github.com/pfaffelh/coins/blob/c60bcd3/CoinsLean/CoinsLean/AboveLimit.lean\#L89}{AboveLimit.lean:89} \\
Theorem~\ref{thm:above-limit} (lower bound) &
  \texttt{above\_limit\_ge} &
  \href{https://github.com/pfaffelh/coins/blob/c60bcd3/CoinsLean/CoinsLean/AboveLimit.lean\#L133}{AboveLimit.lean:133} \\
Proposition~\ref{prop:delta-rec} &
  \texttt{deficit\_succ} &
  \href{https://github.com/pfaffelh/coins/blob/c60bcd3/CoinsLean/CoinsLean/Perturbation.lean\#L301}{Perturbation.lean:301} \\
Proposition~\ref{prop:cn} &
  \texttt{deficit\_first\_order} &
  \href{https://github.com/pfaffelh/coins/blob/c60bcd3/CoinsLean/CoinsLean/Perturbation.lean\#L2503}{Perturbation.lean:2503} \\
Example~\ref{ex:cn-small} ($c_1=1$) &
  \texttt{c\_one} &
  \href{https://github.com/pfaffelh/coins/blob/c60bcd3/CoinsLean/CoinsLean/Perturbation.lean\#L76}{Perturbation.lean:76} \\
Example~\ref{ex:cn-small} ($c_2=\tfrac{3}{2}$) &
  \texttt{c\_two} &
  \href{https://github.com/pfaffelh/coins/blob/c60bcd3/CoinsLean/CoinsLean/Perturbation.lean\#L119}{Perturbation.lean:119} \\
Example~\ref{ex:cn-small} ($c_3=\tfrac{27}{16}$) &
  \texttt{c\_three} &
  \href{https://github.com/pfaffelh/coins/blob/c60bcd3/CoinsLean/CoinsLean/Perturbation.lean\#L169}{Perturbation.lean:169} \\
Example~\ref{ex:cn-small} ($c_4=\tfrac{111}{64}$) &
  \texttt{c\_four} &
  \href{https://github.com/pfaffelh/coins/blob/c60bcd3/CoinsLean/CoinsLean/Perturbation.lean\#L186}{Perturbation.lean:186} \\
Example~\ref{ex:cn-small} ($c_5=\tfrac{3555}{2048}$) &
  \texttt{c\_five} &
  \href{https://github.com/pfaffelh/coins/blob/c60bcd3/CoinsLean/CoinsLean/Perturbation.lean\#L211}{Perturbation.lean:211} \\
Example~\ref{ex:cn-small} ($c_6=\tfrac{113337}{65536}$) &
  \texttt{c\_six} &
  \href{https://github.com/pfaffelh/coins/blob/c60bcd3/CoinsLean/CoinsLean/Perturbation.lean\#L244}{Perturbation.lean:244} \\
Lemma~\ref{lem:collapse} (low) &
  \texttt{suffMin\_collapse\_low} &
  \href{https://github.com/pfaffelh/coins/blob/c60bcd3/CoinsLean/CoinsLean/Perturbation.lean\#L1501}{Perturbation.lean:1501} \\
Lemma~\ref{lem:collapse} (high) &
  \texttt{suffMin\_collapse\_high} &
  \href{https://github.com/pfaffelh/coins/blob/c60bcd3/CoinsLean/CoinsLean/Perturbation.lean\#L1507}{Perturbation.lean:1507} \\
Lemma~\ref{lem:lower} &
  \texttt{c\_ge\_27\_16\_full} &
  \href{https://github.com/pfaffelh/coins/blob/c60bcd3/CoinsLean/CoinsLean/Perturbation.lean\#L1483}{Perturbation.lean:1483} \\
Lemma~\ref{lem:decreasing} &
  \texttt{c\_strict\_anti\_from\_five} &
  \href{https://github.com/pfaffelh/coins/blob/c60bcd3/CoinsLean/CoinsLean/Perturbation.lean\#L1488}{Perturbation.lean:1488} \\
Proposition~\ref{prop:linear} &
  \texttt{c\_linear\_rec} &
  \href{https://github.com/pfaffelh/coins/blob/c60bcd3/CoinsLean/CoinsLean/Perturbation.lean\#L1495}{Perturbation.lean:1495} \\
Identity~\eqref{eq:alg-id} &
  \texttt{alg\_id} &
  \href{https://github.com/pfaffelh/coins/blob/c60bcd3/CoinsLean/CoinsLean/Perturbation.lean\#L932}{Perturbation.lean:932} \\
$\sum_{n\ge 13}B_n<\tfrac{1}{8}$ (\S\ref{sec:joint}) &
  \texttt{B\_tail\_bound} &
  \href{https://github.com/pfaffelh/coins/blob/c60bcd3/CoinsLean/CoinsLean/Perturbation.lean\#L864}{Perturbation.lean:864} \\
$\varepsilon_{12}\ge\tfrac{1}{60}$ (\S\ref{sec:joint}) &
  \texttt{c\_twelve\_buffer} &
  \href{https://github.com/pfaffelh/coins/blob/c60bcd3/CoinsLean/CoinsLean/Perturbation.lean\#L2995}{Perturbation.lean:2995} \\
$\sum_{n\ge 13}\delta_n\le\tfrac{1}{200}$ (\S\ref{sec:joint}) &
  \texttt{delta\_tail\_bound} &
  \href{https://github.com/pfaffelh/coins/blob/c60bcd3/CoinsLean/CoinsLean/Perturbation.lean\#L889}{Perturbation.lean:889} \\
Cumulative $\varepsilon$-bound (\S\ref{sec:joint}) &
  \texttt{cum\_eps\_bound} &
  \href{https://github.com/pfaffelh/coins/blob/c60bcd3/CoinsLean/CoinsLean/Perturbation.lean\#L949}{Perturbation.lean:949} \\
Corollary~\ref{cor:L-exists} &
  \texttt{c\_limit\_exists} &
  \href{https://github.com/pfaffelh/coins/blob/c60bcd3/CoinsLean/CoinsLean/Perturbation.lean\#L1517}{Perturbation.lean:1517} \\
Theorem~\ref{thm:limit} &
  \texttt{c\_limit\_formula} &
  \href{https://github.com/pfaffelh/coins/blob/c60bcd3/CoinsLean/CoinsLean/Perturbation.lean\#L2065}{Perturbation.lean:2065} \\
Corollary~\ref{cor:shape}~(i) &
  \texttt{w\_gap\_first\_order} &
  \href{https://github.com/pfaffelh/coins/blob/c60bcd3/CoinsLean/CoinsLean/Perturbation.lean\#L2811}{Perturbation.lean:2811} \\
Corollary~\ref{cor:shape}~(ii) &
  \texttt{w\_local\_min\_at\_five} &
  \href{https://github.com/pfaffelh/coins/blob/c60bcd3/CoinsLean/CoinsLean/Perturbation.lean\#L2847}{Perturbation.lean:2847} \\
Corollary~\ref{cor:shape}~(iii) &
  \texttt{no\_first\_order\_local\_max} &
  \href{https://github.com/pfaffelh/coins/blob/c60bcd3/CoinsLean/CoinsLean/Perturbation.lean\#L2948}{Perturbation.lean:2948} \\
\hline
\end{tabular}
\caption{Correspondence between manuscript results and their Lean
  formalization.  Links point to the exact line of
  commit~\texttt{c60bcd3} on the repository's
  \texttt{main} branch.}\label{tab:lean-map}
\end{table}

The dominated-convergence step in the proof of
Theorem~\ref{thm:limit} (Step~2) is discharged in Lean by
Mathlib's \href{https://leanprover-community.github.io/mathlib4_docs/Mathlib/Analysis/Normed/Group/Tannery.html}{\texttt{tendsto\_tsum\_of\_dominated\_convergence}}
(Tannery's theorem), and the infinite-product convergence in
Step~1 uses
\href{https://leanprover-community.github.io/mathlib4_docs/Mathlib/Analysis/SpecificLimits/Normed.html}{\texttt{Real.multipliable\_one\_add\_of\_summable}}.
These are exactly the two analytic ingredients cited in the proof
of Theorem~\ref{thm:limit}.

\paragraph{Independent verification with the comparator.}
A referee who wants to mechanically verify the formalization
without auditing the full proof modules can do so using the Lean
comparator~\cite{LeanComparator}, a sandboxed kernel-replay tool
maintained by the Lean FRO.  The trust surface is then the two
short files
\href{https://github.com/pfaffelh/coins/blob/main/CoinsLean/Challenge.lean}{\texttt{CoinsLean/Challenge.lean}}
(the sixteen theorem statements with \texttt{sorry} placeholders)
and
\href{https://github.com/pfaffelh/coins/blob/main/CoinsLean/CoinsLean/Defs.lean}{\texttt{CoinsLean/CoinsLean/Defs.lean}}
(the seven shared definitions \texttt{a}, \texttt{w}, \texttt{c},
\texttt{deficit}, \texttt{suffMin}, \texttt{A\_lin}, \texttt{B\_lin}
on which the statements depend).  A successful comparator run
guarantees that the proofs in
\href{https://github.com/pfaffelh/coins/blob/main/CoinsLean/Solution.lean}{\texttt{Solution.lean}}
prove exactly the statements declared in \texttt{Challenge.lean},
use only the three foundational axioms above, and are accepted by
the Lean kernel.  Concretely, after installing the three external
binaries \texttt{landrun}, \texttt{lean4export}, and
\texttt{comparator} (build instructions in the project's
\texttt{README.md}), one runs from inside \texttt{CoinsLean/}
\begin{center}
  \texttt{rm -rf .lake/build \&\& lake env comparator config.json}
\end{center}
where \texttt{config.json} (also at the project root) lists the
sixteen theorem names and the permitted axioms.

\paragraph{Development history.}
The formalization was produced in two stages of interactive work
with Anthropic's Claude.  An earlier session using Claude~Opus~4.6
established the initial scaffolding of the project --- the
Mathlib integration, the strategy-\All{} file
\texttt{Bellman.lean}, and the fair-coin computation in
\texttt{HalfP.lean} (totalling about 100 lines).  The present
paper's formalization --- the optimal-value definition
\texttt{Optimal.lean}, the strategies module, the full proof of
Theorem~\ref{thm:above} in \texttt{Above.lean}, the limit
analysis above $\tfrac12$ in \texttt{AboveLimit.lean}, and the
entire \S4 perturbation module \texttt{Perturbation.lean} (about
$3{,}700$ lines of Lean total) --- was produced in a subsequent
sustained interaction with Claude~Opus~4.7, with the author as
mathematical guide and reviewer.  Each step was decided through
back-and-forth: the author selected which result to tackle next
and acted as reviewer; Claude proposed, drafted, debugged, and
revised proofs until they closed with \texttt{lake build} green and
no \texttt{sorry}.  A full prompt-by-prompt log of the second
session is preserved in
\href{https://github.com/pfaffelh/coins/blob/main/CONVERSATION_LOG.md}{\texttt{CONVERSATION\_LOG.md}},
and a thematic commentary tied to each commit is in
\href{https://github.com/pfaffelh/coins/blob/main/journal.md}{\texttt{journal.md}}.

\end{document}